\documentclass[10pt,letter]{article}
\title{\Large \textbf{A REFINED CONJECTURE FOR FACTORING ITERATES OF
QUADRATIC POLYNOMIALS
OVER FINITE FIELDS}}

\author{\sc Vefa Goksel\thanks{\small Department of Mathematics, Bilkent
University, 06800
Ankara,Turkey,
\texttt{\href{mailto:v$_-$goksel@ug.bilkent.edu}{v$_-$goksel@ug.bilkent.edu.tr}}}
\and
\sc Shixiang Xia\thanks{\small Olin School of Business, Washington
University, Saint Louis, Missouri 63130,
\texttt{\href{mailto:sxia@wustl.edu}{sxia@wustl.edu}}}
\and
\sc Nigel Boston\thanks{\small Department of Mathematics, University of
Wisconsin,
Madison, Wisconsin 53706,
\texttt{\href{mailto:boston@math.wisc.edu}{boston@math.wisc.edu}}}
}

\date{\small Mathematics Subject Classifications: Primary 11T55, 37P25, 60J20.}

\usepackage{amsfonts,amsmath,amssymb,amsthm,mathrsfs}
\usepackage{float}
\usepackage{amssymb}
\newtheorem{thm}{Theorem}[section]
\newtheorem{lemma}[thm]{Lemma}
\newtheorem{conj}[thm]{Conjecture}
\newtheorem{corollary}[thm]{Corollary}
\newtheorem{observation}[thm]{Observation}

\newtheorem{remark}[thm]{Remark}
\theoremstyle{definition}
\newtheorem{definition}{Definition}[section]
\newtheorem{example}[definition]{Example}

\theoremstyle{remark}

\usepackage{bibentry}
\usepackage{graphicx}
\usepackage{titlesec}
\titlelabel{\thesection.\space}

\def\F{\mathbb{F}}

\def\N{\mathbb{N}}

\usepackage{fancyhdr}

\usepackage{hyperref}
\hypersetup{pdftitle={A REFINED CONJECTURE FOR FACTORING ITERATES OF
QUADRATIC POLYNOMIALS
OVER FINITE FIELDS}}
\hypersetup{pdfauthor={\copyright\  Vefa Goksel, Shixiang Xia, and Nigel Boston}}
\hypersetup{pdfsubject={Group theory, number theory}}
\hypersetup{pdfkeywords={Quadratic polynomials, finite fields, transition
matrix,
factorization of iterates, limiting distribution}}

\begin{document}

\maketitle

\noindent \textsc{Abstract.} Jones and Boston conjectured that
the factorization process for iterates of irreducible quadratic
polynomials over finite fields is approximated by a  Markov model. In this
paper, we find unexpected and intricate  
behavior for some quadratic polynomials, in particular for the ones with tail size one. We also
propose a multi-step Markov model that explains these new observations  
better than the model of Jones and Boston.
\section{Introduction}
 Let $f$ be an irreducible quadratic polynomial over a finite field $\F_q$ of
odd order $q$. We are interested in understanding the factorization of
iterates of $f$. This problem was previously studied in \cite{perez}, \cite{gomez}, \cite{ahmadi}, \cite{ayad}, and  \cite{jones+boston}. In \cite{jones+boston}, the authors  associated a Markov process to $f$ and
conjectured that its limiting distribution explains 
the shape of the factorization of large iterates of $f$. In this paper, we give new data that strongly suggest a  
more complicated model is required in certain cases, and we propose a  
multi-step Markov model that fits the new data well. Furthermore, we also conjecture that the original Markov model applies except in these certain cases.\\

The paper is structured as follows. In Section 2, we make some definitions,
give preliminary results, and recall background to the problem. In Section
3, we provide some examples with new, unexpected behavior. In
Section 4, we propose a multi-step Markov model to describe the factorization of iterates and
conjecture that it provides a better explanation for the process. Section 5
supports this model via actual data corresponding to the examples given in
Section 3. In Section 6, we summarize matters with some further conjectures and list 
additional computational results we have obtained.
\section{Set-up}

\begin{definition}\label{def1}
Let $\F_q$ be a finite field of odd order $q$. Consider a
quadratic polynomial $f(x)$ defined over $\F_q$. For all $n\in\N$, we define the
\emph{$n$th iterate} of
$f$ to be $f^{n}(x):=f({f^{n-1}(x)}).$ We make the convention that $f^{0}(x):=x$.
\end{definition}
For example, suppose $f(x) = x^2+1 \in \F_{7}[x]$. Then, $f^2(x) = f(f(x)) =
x^4+2x^2+2$, $f^3(x)=f^2(f(x))= x^8+4x^6+x^4+x^2+5$, and so on.

\begin{definition}\label{def2}
Let $f(x)=ax^2+bx+c \in\F_q[x]$
($a\neq 0$) and $\alpha=\frac{-b}{2a}$
be the critical point of $f$. The \emph{critical orbit} of $f$ is
the set $\mathcal{O} := \{f^k(\alpha)\ |\ k=1,2,3,\cdots\}$ and
the number of elements of $\mathcal{O}$ is the \emph{orbit size} of $f$, denoted $o$.
\end{definition}
To illustrate the definition of the critical orbit, we consider the
previous
example. The critical point of $f(x) = x^2+1$ is $0$ and $f(0)=1$,
$f^2(0)=2$,
$f^3(0) =5$, $f^4(0)=5$. It follows that
$f^k(0)=5$ for all $k\geq 3$. Therefore, the critical orbit for
$f(x)=x^2+1 \in \F_{7}[x]$ is $\{1,2,5\}$.

\begin{definition}\label{def3}
Let $f$ be a quadratic polynomial
over $\F_q$ and $\alpha$ be the critical point of $f$. We define
the  \emph{tail} of $f$ to be the set
\begin {align*}
\mathcal{T}:= \{f^k(\alpha)\ | k \geq 1, f^i(\alpha)\neq f^k(\alpha) \forall  i\neq
k\}.
\end{align*}
Similarly, we call the number of elements of $\mathcal{T}$ \emph{the tail
size} of $f$ and denote it by $t$.
\end{definition}

Having taken $f(x)=x^2+c$, the critical orbit of
$f(x)\in \F_q[x]$ becomes $\{c, c^2+c, (c^2+c)^2+c,\cdots\}$. 
\begin{definition}\label{def4}
Noting that $f^n(c)$ is the $(n+1)$th element of the critical orbit
of $f(x)=x^2+c$, we define the \emph{difference polynomial}
$p_{a,b}(c)$ to be
$$  \frac{f^a(c)-f^b(c)}{LCM(f^{b+1}(c)-f^b(c),f^{a-1}(c)-f^{b-1}(c),f^{a-2}(c)-f^{b-2}(c),\cdots, f^{a-b}(c)-f^0(c)) } $$
if $a\neq b+1$ and
$$ \frac{f^a(c)-f^b(c)}{{c}LCM(f^{b+1}(c)-f^b(c),f^{a-1}(c)-f^{b-1}(c),f^{a-2}(c)-f^{b-2}(c),\cdots, f^{a-b}(c)-f^0(c)) } $$
if $a=b+1$.

\end{definition}
\begin{remark}\label{Remark}
$x^2+c_0 \in \F_q[x]$ is a quadratic polynomial with orbit size $o$ and tail
size $t$ iff $c_0$ is a root of $p_{o,t}$ in $\F_q$. The
polynomials in the denominators rescue us from an earlier repetition that would
lead to the correct tail size but smaller orbit size or vice versa.

To illustrate this, suppose we want a quadratic
polynomial $x^2+c\in \F_q[x]$ of orbit size 3 and tail size 1. This
immediately yields
$f^3(c)-f^1(c)={c}^8+4{c}^7+6{c}^6+6{c}^5+5{c}^4+2{c}^3=0$.
This is not sufficient, however, because, for instance, if we set $c=-2$
in this equation, it  holds but the critical orbit is only
$\{-2,2\}$. 

In fact, the above octic factors as $c^3 (c+1)^2 (c+2) (c^2+1)$. All
the factors other than $c^2+1$ ($=p_{3,1}(c)$) lead to degenerate cases.

\end{remark}
\begin{definition}\label{def5}
Let $f(x)\in \F_q[x]$ be an irreducible quadratic polynomial with critical orbit $\mathcal{O}$ and $g(x)\in
\F_q[x]$. We define the
\emph{type of $g(x)$} at $\beta$ to be $s$ if $g(\beta)$ is a square in
$\F_q$ and $n$ if it is not a square. The type of $g$ is a string of
length $|\mathcal{O}|$ whose $k$th entry is the type of
$g(x)$ at the $k$th entry of $\mathcal{O}$. The $k$th entry is also called the $k$th
\emph{digit}.
\end{definition}
For instance, given $x^2+1\in \F_7[x]$, consider $g(x)=x^2+2x+2$. Then,
$g(1)=5$, $g(2)=3$, $g(5)=2$, which implies that the type of $g$ is $nns$.
\begin{definition}\label{def6}
Given an irreducible quadratic polynomial $f(x)\in \F_q[x]$ and a
polynomial $g(x)\in \F_q[x]$, we call the factors of $g(f(x))$ the
\emph{children} of $g$. Also, for any natural number $m$, the factors of
$g(f^{m}(x))$ are called the $m-step$ $descendants$ of $g$.
\end{definition}
\begin{definition}\label{def7}
Let $f(x)\in \F_q[x]$ be a quadratic polynomial and $\gamma$ the unique
critical point of $f$. We say $\gamma$ is \emph{periodic} if there
exists an $i\in \mathbb{N}$ s.t. $f^i(\gamma)=\gamma$.
\end{definition}

Next we quote a lemma which is one of the building blocks of
our paper.
\begin{lemma}\label{lemma}
\cite{jones+boston}
Suppose that  $f \in
\F_q[x]$ is
quadratic with critical orbit of length $o$ and all iterates separable. 
Let $g \in
\F_q[x]$ be irreducible of even degree.  Suppose that $h_1h_2$ is a
non-trivial
factorization of $g(f(x))$, and let $d_i$ (resp. $e_i$) be the $i$th
digit of the
type of $h_1$ (resp. $h_2$).  Then there is some $k$, $1 \leq k \leq o $,
with $d_o
= e_k$ and $e_o = d_k$.  Moreover, $k = o$ if and only if $\gamma$ is
periodic, and in the case $\gamma$ is not periodic, we have $k=t$, where
$t$ is the tail size of $f$.
\end{lemma}
In \cite{jones+boston}, Jones and Boston tried to explain the distribution of types of factors (weighted by their degree) of
iterates of $f$ by a Markov
model as follows: \\

We create a time-homogeneous Markov process $Y_1, Y_2, \ldots$ related
to $f$.
The state space
is the space of types of $f$, namely $\{n, s\}^{o}$, ordered lexicographically. 
We define the Markov process by giving its transition matrix $M =
(\mathcal{P}(Y_m = T_j
| Y_{m-1} = T_i))$, where $T_i$ and $T_j$ vary over all types.  
Note that the entries of each column of $M$ sum to $1$.  We
define $M$
by assuming that all \textit{allowable types} of children
arise with equal
probability.  To define allowable type, note that $f$ acts on its
critical orbit, and thus also on the set of types.  Indeed, if $T$ is a
type, then
$f(T)$ is obtained by shifting each entry one position to the left and
using the
former $m$th entry as the new final entry, where
$m$ is such that $f^{o+1}(\gamma) = f^m(\gamma)$.  If $g$ has type
$T$ which begins with $n$, then $g$ has only one child, 
and it will have type $f(T)$, the only allowable type in
this case.
If $T$ begins with $s$, then $g$ has two children, whose types have product
$f(T)$. Among pairs of types $T_1, T_2$ with $T_1T_2 = f(T)$, we call
allowable
those that satisfy the conclusion of Lemma 2.3, namely  $d_k =
e_o$ and $e_k = d_o$ with $k = o$ if $\gamma$ is periodic, and $k=t$ if
$\gamma$ is not periodic, where $t$ is the tail size of $f$.
\begin{definition}\label{def8}
We define an \emph{m-step transition matrix} as
$M_m=(\mathcal{P}(Y_{m+1} = T_j| Y_1=T_i)$ 
by assuming that all allowable choices of m-step transition
arise with equal probablity. Here, allowable refers to those that arise for the given $f$,
which turns out to be a subtle matter at the heart of this paper.
\end{definition}
\begin{remark}\label{remark}
Since transitions in a Markov model are
independent of each other, the model of Jones and Boston \cite{jones+boston} implies
that $M_m={M_1}^m$ always holds.
\end{remark}
\section{New Phenomena}
Contrary to what Jones and Boston suggested, we discover that
the story of these descendants can be quite different in certain cases. More
precisely, in these special cases, not every 2-step or 3-step
transition permitted by the above model actually occurs. Thus, a Markov model
does not apply to all quadratic polynomials. We now illustrate this idea
with three kinds of examples:
\begin{example}\label{eg1} The first kind has orbit size 3
and tail size 1. As computed earlier, $p_{3,1}(c)=c^2+1$,
so these are the quadratic polynomials of the form $f(x) = x^2+i$, where $i$ is a
square root of $-1$ in $\F_q$. (Note that to do so, we need $q \equiv 1
\pmod 4$ and in fact $q \equiv 5 \pmod 8$ to ensure that $f(x)$ is irreducible).
 The critical orbit is $\{i, i-1, -i\}$. Using Lemma 2.3, the
following 1-step transitions arise:
\begin{align*}
&nnn \mapsto nnn\\
&nns \mapsto nsn\\
&nsn \mapsto sns\\
&nss \mapsto sss\\
&snn \mapsto nns/ssn\ \text{or}\ nss/snn\\
&sns \mapsto nns/snn\ \text{or}\ nss/ssn\\
&ssn \mapsto nnn/nsn\ \text{or}\ sns/sss\\
&sss \mapsto nnn/nnn\ \text{or}\ nsn/nsn\ \text{or}\ sns/sns\ \text{or}\
sss/sss.
\end{align*}
It follows that 
\[M_1 = 
\begin{bmatrix}
1 & 0 & 0 & 0 & 0 & 0 & 1/4 & 1/4 \\
0 & 0 & 0 & 0 & 1/4 & 1/4 & 0 & 0 \\
0 & 1 & 0 & 0 & 0 & 0 & 1/4 & 1/4 \\
0 & 0 & 0 & 0 & 1/4 & 1/4 & 0 & 0 \\
0 & 0 & 0 & 0 & 1/4 & 1/4 & 0 & 0 \\
0 & 0 & 1 & 0 & 0 & 0 & 1/4 & 1/4 \\
0 & 0 & 0 & 0 & 1/4 & 1/4 & 0 & 0 \\
0 & 0 & 0 & 1 & 0 & 0 & 1/4 & 1/4 
\end{bmatrix}
\]
The new phenomenon is that the following was observed.
\begin{observation}\label{Observation}
Let $q \equiv 5 \pmod 8$.
Let $f(x)=x^2+i\in \F_q[x] $, where $i$ is a
square root of $-1$.
Then, the following 2-step transitions never occur:
\begin{align*}
&nsn \mapsto nns/snn\\
&nss \mapsto nnn/nnn\\
&nss \mapsto sns/sns.\\
\end{align*}
\end{observation}

In particular, $M_2\neq {M_1}^{2}$. That is to say, there is a
discrepancy between the proposed Markov model and what actually happens. Since we
know which 2-step transitions are forbidden, we explicitly calculate
the discrepancy matrix $A:= M_2-{M_1}^2$.
\[A=
\begin{bmatrix}
  0 & 0 & 0 & -1/4 & 0 & 0 & 0 & 0 \\
  0 & 0 & -1/4 & 0 & 0 & 0 & 0 & 0 \\
  0 & 0 & 0 & 1/4 & 0 & 0 & 0 & 0 \\
  0 & 0 & 1/4 & 0 & 0 & 0 & 0 & 0 \\
  0 & 0 & -1/4 & 0 & 0 & 0 & 0 & 0 \\
  0 & 0 & 0 & -1/4 & 0 & 0 & 0 & 0 \\
  0 & 0 & 1/4 & 0 & 0 & 0 & 0 & 0 \\
  0 & 0 & 0 & 1/4 & 0 & 0 & 0 & 0
\end{bmatrix}
\]
\end{example}
\begin{example}\label{eg2}
The second kind has orbit size 4 and tail size 1. We have
$p_{4,1}(c)=c^6+2c^5+2c^4+2c^3+c^2+1$. Let $c_0$ be a root of $p_{4,1}$ in
some $\F_q$ such that $f(x)=x^2+c_0$ is irreducible.  Again applying Lemma 2.3, the
following
1-step transitions are valid:\\ \\
$nnnn \mapsto nnnn$\\
$nnns \mapsto nnsn$\\
$nnsn \mapsto nsnn$\\
$nnss \mapsto nssn$\\
$nsnn \mapsto snns$\\
$nsns \mapsto snss$\\
$nssn \mapsto ssns$\\
$nsss \mapsto ssss$\\
$snnn \mapsto nnns/sssn$ or $nnss/ssnn$ or $nsns/snsn$ or $nsss/snnn$\\
$snns \mapsto nnns/ssnn$ or $nnss/sssn$ or $nsns/snnn$ or $nsss/snsn$\\
$snsn \mapsto nnns/snsn$ or $nnss/snnn$ or $nsns/sssn$ or $nsss/ssnn$\\
$snss \mapsto nnns/snnn$ or $nnss/snsn$ or $nsns/ssnn$ or $nsss/sssn$\\
$ssnn \mapsto nnnn/nssn$ or $nnsn/nsnn$ or $snns/ssss$ or $snss/ssns$\\
$ssns \mapsto nnnn/nsnn$ or $nnsn/nssn$ or $snns/ssns$ or $snss/ssss$\\
$sssn \mapsto nnnn/nnsn$ or $nsnn/nssn$ or $snns/snss$ or $ssns/ssss$\\
$ssss \mapsto nnnn/nnnn$ or $nnsn/nnsn$ or $nsnn/nsnn$ or $nssn/nssn$ or
$snns/snns$ or $snss/snss$ or $ssns/ssns$ or $ssss/ssss$.\\

It follows that
\[M_1 =  \left( \begin{array}{cccccccccccccccc}

1 & 0 & 0 & 0 & 0 & 0 & 0 & 0 & 0 & 0 & 0 & 0 & 1/8 & 1/8 & 1/8 & 1/8 \\
0 & 0 & 0 & 0 & 0 & 0 & 0 & 0 & 1/8 & 1/8 & 1/8 & 1/8 & 0 & 0 & 0 & 0 \\
0 & 1 & 0 & 0 & 0 & 0 & 0 & 0 & 0 & 0 & 0 & 0 & 1/8 & 1/8 & 1/8 & 1/8 \\
0 & 0 & 0 & 0 & 0 & 0 & 0 & 0 & 1/8 & 1/8 & 1/8 & 1/8 & 0 & 0 & 0 & 0 \\
0 & 0 & 1 & 0 & 0 & 0 & 0 & 0 & 0 & 0 & 0 & 0 & 1/8 & 1/8 & 1/8 & 1/8 \\
0 & 0 & 0 & 0 & 0 & 0 & 0 & 0 & 1/8 & 1/8 & 1/8 & 1/8 & 0 & 0 & 0 & 0 \\
0 & 0 & 0 & 1 & 0 & 0 & 0 & 0 & 0 & 0 & 0 & 0 & 1/8 & 1/8 & 1/8 & 1/8 \\
0 & 0 & 0 & 0 & 0 & 0 & 0 & 0 & 1/8 & 1/8 & 1/8 & 1/8 & 0 & 0 & 0 & 0 \\
0 & 0 & 0 & 0 & 0 & 0 & 0 & 0 & 1/8 & 1/8 & 1/8 & 1/8 & 0 & 0 & 0 & 0 \\
0 & 0 & 0 & 0 & 1 & 0 & 0 & 0 & 0 & 0 & 0 & 0 & 1/8 & 1/8 & 1/8 & 1/8 \\
0 & 0 & 0 & 0 & 0 & 0 & 0 & 0 & 1/8 & 1/8 & 1/8 & 1/8 & 0 & 0 & 0 & 0 \\
0 & 0 & 0 & 0 & 0 & 1 & 0 & 0 & 0 & 0 & 0 & 0 & 1/8 & 1/8 & 1/8 & 1/8 \\
0 & 0 & 0 & 0 & 0 & 0 & 0 & 0 & 1/8 & 1/8 & 1/8 & 1/8 & 0 & 0 & 0 & 0 \\
0 & 0 & 0 & 0 & 0 & 0 & 1 & 0 & 0 & 0 & 0 & 0 & 1/8 & 1/8 & 1/8 & 1/8 \\
0 & 0 & 0 & 0 & 0 & 0 & 0 & 0 & 1/8 & 1/8 & 1/8 & 1/8 & 0 & 0 & 0 & 0 \\
0 & 0 & 0 & 0 & 0 & 0 & 0 & 1 & 0 & 0 & 0 & 0 & 1/8 & 1/8 & 1/8 & 1/8
\end{array} \right)\]

Analogously to the first example, however, we observe that once more 
certain 2-step transitions are forbidden. More precisely, the following is observed:\\
\begin{observation}\label{Observation}
Let $c_0$ be a root of
$p_{4,1}$ in $\F_q$ and $f(x) = x^2+c_0 \in \F_q[x]$ be irreducible. Then the
2-step transitions
given below never occur:\\ \\
$nsnn \mapsto nnns/ssnn$ \\
$nsnn \mapsto nsns/snnn$\\
$nsns \mapsto nnns/snnn$ \\
 $nsns \mapsto nsns/ssnn$\\
$nssn \mapsto nnnn/nsnn$ \\
 $nssn \mapsto snns/ssns$\\
$nsss \mapsto nnnn/nnnn$ \\
$nsss \mapsto nsnn/nsnn$ \\
$nsss \mapsto snns/snns$ \\
$nsss \mapsto ssns/ssns$.\\
\end{observation}
By the same reasoning as in Example 3.1, we can explicitly calculate
the discrepancy
 matrix $A:= M_2-{M_1}^2$.
\[A= \left( \begin{array}{cccccccccccccccc}
 0 & 0 & 0 & 0 & 0 & 0 & -1/8 & -1/8  & 0 & 0 & 0 & 0 & 0 & 0 & 0 & 0 \\
 0 & 0 & 0 & 0 & -1/8 & -1/8 & 0 & 0  & 0 & 0 & 0 & 0 & 0 & 0 & 0 & 0 \\
 0 & 0 & 0 & 0 & 0 & 0 & 1/8 & 1/8  & 0 & 0 & 0 & 0 & 0 & 0 & 0 & 0 \\
 0 & 0 & 0 & 0 & 1/8 & 1/8 & 0 & 0  & 0 & 0 & 0 & 0 & 0 & 0 & 0 & 0 \\
 0 & 0 & 0 & 0 & 0 & 0 & -1/8 & -1/8  & 0 & 0 & 0 & 0 & 0 & 0 & 0 & 0 \\
 0 & 0 & 0 & 0 & -1/8 & -1/8 & 0 & 0  & 0 & 0 & 0 & 0 & 0 & 0 & 0 & 0 \\
 0 & 0 & 0 & 0 & 0 & 0 & 1/8 & 1/8  & 0 & 0 & 0 & 0 & 0 & 0 & 0 & 0 \\
 0 & 0 & 0 & 0 & 1/8 & 1/8 & 0 & 0  & 0 & 0 & 0 & 0 & 0 & 0 & 0 & 0 \\
 0 & 0 & 0 & 0 & -1/8 & -1/8 & 0 & 0  & 0 & 0 & 0 & 0 & 0 & 0 & 0 & 0 \\
 0 & 0 & 0 & 0 & 0 & 0 & -1/8 & -1/8  & 0 & 0 & 0 & 0 & 0 & 0 & 0 & 0 \\
 0 & 0 & 0 & 0 & 1/8 & 1/8 & 0 & 0  & 0 & 0 & 0 & 0 & 0 & 0 & 0 & 0 \\
 0 & 0 & 0 & 0 & 0 & 0 & 1/8 & 1/8  & 0 & 0 & 0 & 0 & 0 & 0 & 0 & 0 \\
 0 & 0 & 0 & 0 & -1/8 & -1/8 & 0 & 0  & 0 & 0 & 0 & 0 & 0 & 0 & 0 & 0 \\
 0 & 0 & 0 & 0 & 0 & 0 & -1/8 & -1/8  & 0 & 0 & 0 & 0 & 0 & 0 & 0 & 0 \\
 0 & 0 & 0 & 0 & 1/8 & 1/8 & 0 & 0  & 0 & 0 & 0 & 0 & 0 & 0 & 0 & 0 \\
 0 & 0 & 0 & 0 & 0 & 0 & 1/8 & 1/8  & 0 & 0 & 0 & 0 & 0 & 0 & 0 & 0
\end{array} \right)\]
\end{example}
\begin{example}\label{eg3}
Lastly, we consider examples with orbit size 3 and tail size 2. In this
case, the difference polynomial $p_{3,2}(c)=c^3+2c^2+2c+2$. Using 
Lemma 2.3, the 1-step transitions are
as given below:
\begin{align*}
 &nnn \mapsto nnn\\
&nns \mapsto nss\\
&nsn \mapsto sns\\
&nss \mapsto sss\\
&snn \mapsto nsn/sns\ \text{or}\ nns/ssn\\
&sns \mapsto nnn/snn\ \text{or}\ sss/nss\\
&ssn \mapsto nns/nsn\ \text{or}\ sns/ssn\\
&sss \mapsto nnn/nnn\ \text{or}\ nss/nss\ \text{or}\ snn/snn\ \text{or}\
sss/sss.
\end{align*}

It follows that
\[M_1 = 
\begin{bmatrix}
1 & 0 & 0 & 0 & 0 & 1/4 & 0 & 1/4 \\ 
0 & 0 & 0
& 0 & 1/4
& 0 & 1/4 & 0  \\
 0 & 0 & 0 & 0 & 1/4 & 0 & 1/4 & 0  \\ 
 0 & 1 & 0 & 0 & 0 & 1/4 & 0 &
1/4 \\
 0 & 0 & 1 & 0 & 0 & 1/4 & 0 & 1/4 \\  
 0 & 0 & 0 & 0 & 1/4 & 0 & 1/4 &
0  \\
 0 & 0 & 0 & 0 & 1/4 & 0 & 1/4 & 0  \\  
 0 & 0 & 0 & 1 & 0 & 1/4 & 0 & 1/4
 \end{bmatrix}
\]

We observe, however, that certain 3-step transitions never arise.

\begin{observation}\label{Observation}
Let $c_1$ be a root of
$p_{3,2}$ in $\F_q$.
Let $f(x) = x^2+c_1\in \F_q[x]$ be irreducible.
Then the following 3-step
transitions do not occur:
\begin{align*}
 &nns \mapsto nss \mapsto sss \mapsto nss/nss \\
&nns \mapsto nss \mapsto sss \mapsto snn/snn.\\
\end{align*}
\end{observation}
It follows that the discrepancy matrix A, which this time is $M_3-M_{1}^3$, is:\\
\[A= \left( \begin{array}{cccccccc}
 0 & 1/4 & 0 & 0 & 0 & 0 & 0 & 0 \\
 0 & 0 & 0 & 0 & 0 & 0 & 0 & 0 \\
 0 & 0 & 0 & 0 & 0 & 0 & 0 & 0 \\
 0 & -1/4 & 0 & 0 & 0 & 0 & 0 & 0 \\
 0 & -1/4 & 0 & 0 & 0 & 0 & 0 & 0 \\
 0 & 0 & 0 & 0 & 0 & 0 & 0 & 0 \\
 0 & 0 & 0 & 0 & 0 & 0 & 0 & 0  \\
 0 & 1/4 & 0 & 0 & 0 & 0 & 0 & 0   \end{array} \right)\]
\end{example}
\section{New Model}
The investigations in the previous section show
that a Markov model does not always fit the factorization process for
iterates of quadratic polynomials. We need a new model
to explain the process and we propose the following.\\

Let $a-1$ and $b$ be the tail and orbit sizes of an irreducible quadratic
polynomial
$f$ defined over $\F_q$, respectively. Then the $m$-step
transition matrices associated to $f$ satisfy the following recurrence
relation:
\begin{align}
M _{m+a} = M _{m+a-1}B + M_ mA  \label{eq:recurrence}
\end{align}
where $M_{-a+1}=\cdots=M_{-2}=M_{-1}=0$, $M_0=I$.\\
\begin{corollary}\label{corollary}
The following hold for the new model:
\begin{itemize}
\item[(i)] $B=M_1$.
\item[(ii)] $M_i={M_1}^i$ for $i=1,2,\cdots , a-1$.
\item[(iii)] $A= M_a-{M_1}^a$.
\end{itemize}
\end{corollary}
\begin{proof}\label{proof}
\begin{itemize}
\item[(i)] Setting $m=-a+1$ in $(1)$ gives the result.
\item[(ii)]We prove this by induction. Assume
$M_k={M_1}^k$ is true for an integer $1\leq k<a-1$. 
Setting $m=k-a+1$ in $(1)$ gives $M_{k+1}=M_kB+M_{k-a+1}A$.
Since $M_{k-a+1}=0$, $M_{k+1}=M_kB=M_kM_1$. 
By induction, $M_k={M_1}^k$, which yields $M_{k+1}={M_1}^{k+1}$. 
\item[(iii)] Setting $m=0$ in $(1)$ gives $M_a=M_{a-1}B+M_0A$. From the
previous two parts and the initial conditions, we know $B=M_1$,
$M_{a-1}={M_1}^{a-1}$, and $M_0=I$. Plugging these into (1), the
result follows.
\end{itemize}
\end{proof}
\begin{remark}
The new model with tail size $a-1$ is called an $a$-step Markov model.
\end{remark}
\begin{conj}
The multi-step Markov model given above describes the factorization process for the iterates of an irreducible quadratic
polynomial over a finite field of odd order.
\end{conj}

Of course, this is only approximate at any finite level, but it leads to predictions as regards the
limiting behavior. In particular, the multi-step Markov model predicts that in the limit 100\% of the factorization
of the iterates will be of type $nn \cdots n$ (the unique sink) and also allows us to compute the
limiting relative proportions of the other types as follows.

We fix an arbitrary natural number $m$ and define the vector $v_i$ to be
the vector whose entries are the proportions of all $2^b$ types (lexicographically ordered) for the $(m+i)$th
iterate of the polynomial f.
Say $v=(v_1,v_2,\cdots,v_a)$. Then, using (1), the next such $a$-tuple will, according to the model,
be the vector $(v_2,
v_3,\cdots v_a, Av_1+Bv_a)$. Denoting the associated $a2^b$ by $a2^b$
transition matrix by T, we have 
\begin{align*}
T =
 \begin{pmatrix}
  0 & I  & \cdots & 0 \\
  \vdots &  & \ddots & 0 \\
   \vdots &   &  & I \\
  A &0  &\cdots  & B
 \end{pmatrix}.
\end{align*}

We can thereby interpret this multi-step Markov model as a Markov process on a larger number of states,
with transition matrix $T$. The limiting frequencies of the non-absorbing states are
given, up to scaling, by the entries of an eigenvector of $T$ corresponding to its largest
eigenvalue less than $1$. \cite{seneta}

Combining this fact with the following lemma indicates
how the limiting proportions can be computed:
\begin{lemma}\label{Lemma}
With the notation as above, let $e$ be an eigenvector of the
transition matrix $T$ corresponding to eigenvalue $\lambda$, and $e_1$ be its first $2^b$ entries.
Then $e=(e_1, \lambda e_1, \lambda^2 e_1,\cdots,
\lambda^{k-1} e_1)$
\end{lemma}
\begin{proof}\label{proof}
This is a consequence of Theorem 3.2 in \cite{Dennis} (or can be easily directly proven).
\end{proof}
Again with the notation above, consider the eigenvector
$e$ of $T$, corresponding to the largest eigenvalue less than $1$, such that the entries of $e_1$ except the first one sum to 1. 
The entries of $e_1$ are the limiting
proportions of the types that are not $nn\cdots n$.

\section{Data}
In this section, we provide actual data corresponding to examples
3.1, 3.2
and 3.3. In each case, we use the smallest $q$ for which the corresponding
difference polynomial has a root and that yields an irreducible quadratic.
Comparing the limiting proportions predicted by
the new model with the data for each example, we will illustrate
how well the multi-step Markov model fits.
\subsection*{Data for Example 3.1}
\begin{flushleft}
\begin{tabular}{|l|l|l|l|l|l|l|l|}
  \multicolumn{8}{c}{} \\
  \hline
  Iterate & nns & nsn & nss & snn & sns & ssn & sss \\
  \hline
  20 & 0.0251& 0.1748 & 0.1163 &0.0271 &0.2541 &0.1143 &0.2883\\
  21 & 0.0268 &0.1661 & 0.1221 &0.0267 &0.2635 &0.1222 &0.2726\\
  22 & 0.0300 &0.1725 & 0.1253 &0.0271 &0.2487 &0.1282 &0.2681\\
  23 & 0.0256 &0.1689 & 0.1223 &0.0253 &0.2508 &0.1226 &0.2846\\
  24 & 0.0238 &0.1686 & 0.1240 &0.0238 &0.2542 &0.1239 & 0.2817\\
  25 & 0.0276 &0.1669 & 0.1217 &0.0272 &0.2598 &0.1220 & 0.2748\\
  26 & 0.0263 & 0.1699 & 0.1276 &0.0282 &0.2526 &0.1256 &0.2697\\
  27 & 0.0263 & 0.1677 & 0.1237 &0.0269 &0.2502 &0.1231 &0.2821\\
  \hline
\end{tabular}
\end{flushleft}
Table 1: Relative proportions of types (other than $nnn$) for factors of
iterates of
$f(x)=x^2+2 \in \F_5[x]$.\\

By comparison, if we consider the related block matrix 
in the previous section, the first part $e_1$ of an eigenvector for
the eigenvalue $\lambda\approx 0.9333801995$ is
\[
\begin{bmatrix}
-1.0000000000\cdots\\
0.026110931\cdots\\
0.170493119\cdots\\
0.123960675\cdots\\
0.026110931\cdots\\
0.254036800\cdots\\
0.123960675\cdots\\
0.275326866\cdots
\end{bmatrix}
\]
\subsection*{Data for Example 3.2}
\begin{flushleft}
\resizebox{16cm}{!}{
\begin{tabular}{|l|l|l|l|l|l|l|l|l|l|l|l|l|l|l|l|}
  \multicolumn{6}{c}{} \\
  \hline
  Iterate & $nnns$ & $nnsn$ & $nnss$ & $nsnn$ & $nsns$ & $nssn$ & $nsss$ &
$snnn$ & $snns$ & $snsn$ & $snss$ & $ssnn$ & $ssns$ & $sssn$ & $ssss$ \\
  \hline
  21 & 0.0180 &0.0932 & 0.0446 &0.0809 &0.0203 &0.1194 &0.0536& 0.0129
&0.1114 & 0.0501 &0.0845 &0.0230 &0.1227 &0.0505&0.1152\\
  22 & 0.0177 &0.0705 & 0.0483 &0.1086 &0.0187 &0.1039 &0.0483& 0.0137
&0.1021 & 0.0486 &0.0811 &0.0210 &0.1450 &0.0497&0.1228\\
  23 & 0.0178 &0.0816 & 0.0414 &0.0934 &0.0182 &0.1135 &0.0476& 0.0180
&0.1305 & 0.0465 &0.0870 &0.0171 &0.1272 &0.0435&0.1166\\
  24& 0.0232 &0.0804 & 0.0493 &0.1044 &0.0189 &0.0992 &0.0524& 0.0183
&0.1116 & 0.0559 &0.0763 &0.0169 &0.1348 &0.0527&0.1057\\
  25& 0.0190 &0.0859 & 0.0469 &0.1007 &0.0191 &0.1138 &0.0486& 0.0185
&0.1254 & 0.0487 &0.0769 &0.0199 &0.1187 &0.0464&0.1114\\
  26 & 0.0188 &0.0739 & 0.0486 &0.1056 &0.0199 &0.1020 &0.0500& 0.0173
&0.1217 & 0.0493 &0.0776 &0.0194 &0.1332 &0.0514&0.1115\\
  27 & 0.0178 &0.0828 & 0.0497 &0.0963 &0.0189 &0.1107&0.0493& 0.0176
&0.1266 & 0.0505 &0.0792 &0.0186 &0.1218 &0.0489&0.1115\\
 \hline
\end{tabular}
}
\end{flushleft}
Table 2: \small Relative proportions of types (other than $nnnn$) for
factors of iterates of $f(x)=x^2+3 \in \F_{11}[x]$.\\

If we compute the appropriate eigenvector of the related $32$ by $32$
matrix, its first block $e_1$ of size $16$ is
\[
\begin{bmatrix}
 -1.0000000000\cdots\\
0.018669399\cdots\\
0.079050806\cdots\\
0.049246267\cdots\\
0.099196036\cdots\\
0.018669399\cdots\\
0.110198525\cdots\\
0.049246267\cdots\\
0.018669399\cdots\\
0.119717366\cdots\\
0.049246267\cdots\\
0.079050806\cdots\\
0.018669399\cdots\\
0.130925265\cdots\\
0.049246267\cdots\\
0.110198525\cdots
\end{bmatrix}
\]
\subsection*{ Data for Example 3.3}
\begin{flushleft}
\begin{tabular}{|l|l|l|l|l|l|l|l|}
  \multicolumn{8}{c}{} \\
  \hline
  Iterate & nns & nsn & nss & snn & sns & ssn & sss \\
  \hline
  26 & 0.0731 &0.0728 &0.1673 &0.1827 &0.0718 &0.0722 &0.3601\\
  27 & 0.0760 &0.0727 &0.1695 &0.1863 &0.0699 &0.0732 &0.3523\\
  28 & 0.0736 &0.0754 & 0.1798 &0.1734 &0.0747 &0.0729 &0.3502\\
  29 & 0.0654 &0.0761 & 0.1639 &0.1873 &0.0772 &0.0665 &0.3636\\
  30 & 0.0747 &0.0762 & 0.1757 &0.1876 &0.0730 &0.0714 &0.3414\\
  31 &0.0714 &0.0772 & 0.1735 &0.1772 &0.0766 &0.0707 & 0.3535\\
  32 & 0.0715 &0.0713 & 0.1818 &0.1910 &0.0703 &0.0706 & 0.3434\\
  33 & 0.0716 &0.0756  &0.1720  &0.1738 &0.0783 &0.0743 &0.3544\\
  34 & 0.0711 &0.0708 &0.1859  &0.1863 &0.0715 &0.0718 &0.3426\\

  \hline
\end{tabular}
\end{flushleft}
Table 3: Relative proportions of types (other than $nnn$) for factors of
iterates of
$f(x)=x^2+1 \in \F_7[x]$.\\

As mentioned before, in \cite{jones+boston}, Jones and Boston proposed a
Markov process, and they supported this claim
by the example $x^2+1$ over $\F_7$. However, the result  given in Observation 3.3 does not follow this claim.
To illustrate how the multi-step Markov model
fits better, the following table compares the limiting proportions predicted by the Markov model and the multi-step Markov model:
\begin{flushleft}
\begin{tabular}{|l|l|l|l|l|l|l|l|}
  \multicolumn{3}{c}{} \\
  \hline
   Types & Markov model & Multi-step Markov model \\
  \hline
  $nns$ & $0.073573805\cdots$ &$0.071981460\cdots$  \\
  $nsn$ & $0.073573805\cdots$ &$0.071981460\cdots$  \\
  $nss$ & $0.191577027\cdots$ & $0.178322872\cdots$\\
  $snn$ &  $0.191577027\cdots$&$0.178322872\cdots$ \\
  $sns$ & $0.073573805\cdots$ &$0.071981460\cdots$ \\
  $ssn$ &$0.073573805\cdots$ &$0.071981460\cdots$\\
  $sss$ & $0.322550722\cdots$ &$0.355428413\cdots$\\
  \hline
\end{tabular}
\end{flushleft}
Table 4: Limiting proportions of types (other than $nnn$) for factors of
iterates of $x^2+1 \in {\F_7[x]}$ predicted by the Markov model and the multi-step Markov
model.   \\

It is particularly striking how much better the new model fits the data for $sss$.

\section{Conjectures/Speculations}
In this last part, we present some conjectures based on the many
different computational results we have obtained.

In section 3, we observed that for the irreducible quadratic polynomials with
difference polynomials $p_{3,1}$ and $p_{4,1}$, there are
certain missing
2-step transitions. After further investigations with many quadratic
polynomials, we conjecture that the same phenomenon happens for every irreducible quadratic polynomial
with tail size $1$. What would establish that those 2-step transitions are forbidden is the following conjecture.
\begin{conj}\label{conjecture}
Let $f$ be an irreducible quadratic
polynomial over $\F_q$ with tail size $t=1$ and orbit size $o$ and let $g$ be an even irreducible
polynomial over $\F_q$ whose type begins with $ns$. Then, the $(o-1)$th
digit of the type of each irreducible factor of $g(f(x))$ is $s$.
\end{conj}

\begin{example} 
Note that the $o$th digit is $-c$ and so the $(o-1)$th digit is $\alpha$ where $\alpha^2+c = -c$, i.e.
$\alpha^2 = -2c$. Suppose that $g(x) = x^4+ax^2+b$. Then $g(x^2+c)$ factors as $h(x)h(-x)$ ($\ast$)
and we must show that $h(\alpha)$ is a square. If $h(x) = x^4+px^3+qx^2+rx+s$, then, comparing coefficients on the two sides of 
$(\ast)$, we eliminate $q,s,a$, leaving that 
$$h(\alpha) = (\alpha^2+p\alpha/2+r/p)^2 = (-2c+r/p+p\alpha/2)^2.$$
\end{example}

\begin{remark}\label{Remark}
The above conjecture applies to the case $f(x)=x^2-2$, too, which is the simplest
with tail size $1$.
It is, however, vacuous for factors of iterates of $x^2-2$
itself, because, as indicated by Jones and Boston \cite{jones+boston},
the factors are entirely of type $nn$ after a finite number of iterates, whatever $q$ is.
\end{remark}

We end by listing other cases investigated, not covered in
previous sections:
\begin{itemize}
\item[(i)]$o=4$.\\ $t=2$.\\ $p_{4,2}(c)=c^3 + c^2 - c + 1$.\\The first
example is $x^2+4\in \F_7[x]$.\\This appears to have no missing transitions, i.e. follows a Markov model.
\item[(ii)]$o=5$.\\ $t=2$.\\ $p_{5,2}(c)=c^{12} + 6c^{11} + 14c^{10} + 18c^9
+ 18c^8 + 16c^7 + 10c^6 + 6c^5 + 5c^4+ 2c^3 + 1$. \\ The first example is
$x^2+12\in \F_{17}[x]$.\\ This appears to have no missing transitions, i.e. follows a Markov model.
\item[(iii)]$o=4$.\\$t=3$.\\ $p_{4,3}(c)=c^7 + 4c^6 + 6c^5 + 6c^4 + 6c^3 +
4c^2 + 2c + 2$. \\ The first example is $x^2+2\in \F_7[x]$. \\ This appears to have no missing transitions, i.e. follows a Markov model.
\item[(iv)]$o=5$.\\$t=3$.\\$p_{5,3}(c)=c^8 + 4c^7 + 6c^6 + 6c^5 + 4c^4 + 1$.
\\ The first example is $x^2+1\in \F_{11}[x]$.\\ This appears to have no missing transitions, i.e. follows a Markov model.
\end{itemize}
The evidence so far suggests that the only cases where a Markov process does not hold are
those noted earlier, namely tail size $1$ and any orbit size or tail size $2$ and orbit size $3$. 
\begin{conj}\label{conjecture}
Let $f(x)$ be a
quadratic irreducible polynomial in $\F_q[x]$ with orbit size $o$ and tail
size $t$. Then the Markov model fits the factorization process for iterates
of $f$ if and only if $(o,t)\not \in  \left\{ {(m,1)| m \geq 2}\right\}\cup
\left\{ {(3,2)}\right\} $.
\end{conj}
\section*{Acknowledgements}
 The authors owe Rafe Jones a debt of gratitude for his valuable comments on this work during the preparation of the article.
\bibliographystyle{apacite}

\begin{thebibliography}{12}
\bibitem[Ahmadi et al, 2012]{ahmadi} Omran Ahmadi, Florian Luca, Alina Ostafe and Igor E. Shparlinski, \emph{On Stable Quadratic Polynomials}, Glasgow Mathematical Journal \textbf{54} (2012), no.~2, 359--369.
\bibitem[Ayad and McQuillan, 2000]{ayad}
Mohamed Ayad and Donald~L. McQuillan, \emph{Irreducibility of the iterates of a
 quadratic polynomial over a field}, Acta Arith. \textbf{93} (2000), no.~1,
  87--97.
\bibitem[Dennis, 1976]{Dennis} J.E. Dennis Jr, JF Traub, and RP Weber, \emph{The Algebraic Theory of
Matrix Polynomials}, SIAM Journal on Numerical Analysis and Applications
\textbf{13} (1976),
  no.~6, 831--845.
\bibitem[Gomez-Perez et al, 2011]{gomez} Domingo Gomez-Perez, A. P. Nicolas, A. Ostafe and D. Sadornil, \emph{Stable Polynomials over Finite Fields}, Preprint, 2011.
\bibitem[Gomez-Perez et al, 2012]{perez} Domingo Gomez-Perez, Alina Ostafe, and Igor E. Shparlinski, \emph{On Irreducible Divisors of Iterated Polynomials}, Preprint, 2012.
\bibitem[Jones and Boston, 2012]{jones+boston} Rafe Jones and Nigel Boston, \emph{Settled Polynomials over Finite
Fields}, Proc. Amer. Math. Soc. \textbf{140} (2012),
  no.~6, 1849--1863.
 \bibitem[Seneta, 1981]{seneta}
E.~Seneta, \emph{Non-negative matrices and {M}arkov chains}, Springer Series in
  Statistics, Springer, New York, 2006, Revised reprint of the second (1981)
  edition [Springer-Verlag, New York; MR0719544]. 
\end{thebibliography}

\end{document}